\documentclass[reqno,12pt]{amsart}
\usepackage{amsmath,amsthm,amssymb}
\usepackage{fullpage}
\usepackage{xcolor}
\usepackage{stmaryrd}
\newtheorem{thm}{Theorem}[section]
\newtheorem{lem}[thm]{Lemma}

\theoremstyle{definition}
\newtheorem{defn}[thm]{Definition}

\newtheorem{ex}[thm]{Example}

\theoremstyle{remark} \numberwithin{equation}{section}

\newcommand{\ind}{\operatorname{ind_{{}_{FP}}}}
\begin{document}
\title[]{Positive Solution for a Hadamard Fractional Singular Boundary Value Problem of Order $\mu\in(2,\,3)$}
\date{}
\author{Naseer Ahmad Asif}
\address{Department of Mathematics, School of Science, University of Management and Technology, C-II Johar Town, 54770 Lahore, Pakistan}%
\email{naseerasif@yahoo.com}%

\begin{abstract}
{In this article, we establish the existence of positive solution for the following Hadamard fractional singular boundary value problem
\begin{align*}
{}^{H}D_{a^{+}}^{\,\mu}x(t)+f(t,x(t))&=0,\hspace{0.4cm}t\in(a,\,b),\hspace{0.4cm}0<a<b<\infty,\hspace{0.4cm}2<\mu<3,\\
x(a)=a\,x'(a)=x(b)&=0,
\end{align*}
where $f:(a,\,b)\times(0,\infty)\rightarrow(0,\infty)$ is continuous and singular at $t=a$, $t=b$ and $x=0$. Further, ${}^{H}D_{a^{+}}^{\,\mu}$ is Hadamard fractional derivative of order $\mu$.}
\end{abstract}

\maketitle

\section{Introduction}

In this article, we are concerned with the existence of positive solution for the following Hadamard fractional singular boundary value problem (SBVP)
\begin{equation}\label{hbvp}\begin{split}
{}^{H}D_{a^{+}}^{\,\mu}x(t)+f(t,x(t))&=0,\hspace{0.4cm}t\in(a,\,b),\hspace{0.4cm}0<a<b<\infty,\hspace{0.4cm}2<\mu<3,\\
x(a)=a\,x'(a)=x(b)&=0,
\end{split}\end{equation}
where ${}^{H}D_{a^{+}}^{\,\mu}$ is Hadamard fractional derivative of order $\mu$. Moreover, $f:(a,\,b)\times(0,\infty)\rightarrow(0,\infty)$ is continuous and singular at $t=a$, $t=b$ and $x=0$. We provide sufficient conditions for the existence of positive solution of the Hadamard fractional SBVP \eqref{hbvp} using fixed point index for a completely continuous map defined on a cone. By a positive solution $x$ of the Hadamard fractional SBVP \eqref{hbvp} we mean $x\in C[a,\,b]$, ${}^{H}D_{a^{+}}^{\,\mu}x\in C(a,\,b)$, satisfies \eqref{hbvp} and $x(t)>0$ for $t\in(a,\,b)$.

BVPs involving fractional order differentials have become an emerging area of recent research in science, engineering and mathematics, \cite{miller,podlubny,kilbas,dhelm}. Applying results of nonlinear functional analysis and fixed point theory, many articles have been devoted to the existence of solutions and existence of positive solutions for fractional order BVPs, for details see \cite{szhang,zhaliuwulu,ahmad,jong,agaluc,asif1,wang,alsaluca,xuliubaiwu}. However, most of the articles have been devoted to Riemann-Liouville or Caputo fractional derivatives \cite{szhang,zhaliuwulu,jong,agaluc,wang,alsaluca}, whereas few are devoted to Caputo-Fabrizio or Hadamard fractional derivatives \cite{ahmad,asif1,xuliubaiwu}. During the past few decades great work has been done on the literature of SBVPs, for example \cite{agarwaloregan,asif} are excellent monographs. However, fractional order BVPs having singularity with respect to both time and space variables are few \cite{szhang,asif1,wang}.

Recently, the author \cite{asif1} generalized the definition of Caputo-Fabrizio fractional derivative for an arbitrary order and introduced the notion of Caputo-Fabrizio fractional left and right derivatives denoted by ${}^{CF}D_{0^{+}}^{\,\mu}$ and ${}^{CF}D_{0^{-}}^{\,\mu}$, respectively, and established the existence of symmetric positive solutions for the following Caputo-Fabrizio fractional singular integro-differential BVP
\begin{align*}
(2-\mu)\,{}^{CF}D_{0}^{\,\mu}x(t)+f(t,x(t))&=\left(\frac{\mu-1}{2-\mu}\right)^{2}
\begin{cases}
\int_{t}^{0}e^{-\frac{\mu-1}{2-\mu}(\tau-t)}\,x(\tau)d\tau,\hspace{0.4cm}&t\in(-1,0]\\
\int_{0}^{t}e^{-\frac{\mu-1}{2-\mu}(t-\tau)}\,x(\tau)d\tau,\hspace{0.4cm}&t\in[0,1),
\end{cases}\\
x(\pm1)=x'(0^{\pm})&=0,\hspace{6.3cm}\mu\in(1,2),
\end{align*}
where ${}^{CF}D_{0}^{\,\mu}x(t)={}^{CF}D_{0^{+}}^{\,\mu}x(t)$ for $t\geq0$, ${}^{CF}D_{0}^{\,\mu}x(t)={}^{CF}D_{0^{-}}^{\,\mu}x(t)$ for $t\leq0$. Moreover, $f$ is singular at $t=-1$, $t=1$ and $x=0$.

An important feature of the present manuscript is that Hadamard fractional derivative ${}^{H}D_{a^{+}}^{\,\mu}$ in SBVP \eqref{hbvp} has been considered for an arbitrary $a>0$ and existence of positive solution has been formulated over an arbitrary interval $[a,\,b]\subset(0,\infty)$. Consequently, SBVP of the type \eqref{hbvp} has not been considered before.

The manuscript is organized as follows. In Section \ref{pre}, we recall some definitions from fractional calculus and some preliminary lemmas for construction of the Green's function associated with linear operator in \eqref{hbvp}. Some properties of the Green's function are also presented in the same section. In Section \ref{main}, by using the fixed point index for a completely continuous map on a cone of Banach space and results of functional analysis, we formulate the existence of positive solution in Theorem \ref{mainth}. Further, we give an example to illustrate our main theorem.

\section{Preliminaries}\label{pre}

In this section, we shall state some necessary definitions and preliminary lemmas. The following definitions and lemmas are known \cite{kilbas}.

\begin{defn}\cite{kilbas}
The Hadamard fractional left integral of order $\mu>0$ of a function $x:[a,\infty)\rightarrow\mathbb{R}$, $a>0$, is defined as
\begin{align*}
{}^{H}I_{a^{+}}^{\,\mu}\,x(t)=\frac{1}{\Gamma(\mu)}\int_{a}^{t}\left(\ln\frac{t}{\tau}\right)^{\mu-1}x(\tau)\frac{d\tau}{\tau},\hspace{0.4cm}t\geq a.
\end{align*}
\end{defn}

\begin{defn}\cite{kilbas}
The Hadamard fractional left derivative of a function $x:[a,\infty)\rightarrow\mathbb{R}$, $t^{n-1}x^{(n-1)}(t)\in AC[a,\infty)$, $a>0$, $n\in\mathbb{N}$, of order $\mu\in(n-1,n)$ is defined as
\begin{align*}{}^{H}D_{a^{+}}^{\,\mu}x(t)=\frac{1}{\Gamma(n-\mu)}\left(t\frac{d}{dt}\right)^{n}\int_{a}^{t}\left(\ln\frac{t}{\tau}\right)^{n-\mu-1}x(\tau)\frac{d\tau}{\tau},\hspace{0.4cm}t>a.\end{align*}
\end{defn}

\begin{lem}\cite{kilbas}
For $\mu\in(n-1,n)$, $n\in\mathbb{N}$, $y\in L[a,\infty)$, $a>0$, the Hadamard fractional differential equation ${}^{H}D_{a^{+}}^{\,\mu}x(t)+y(t)=0$, $t>a$, has a general solution
\begin{align*}
x(t)=\sum_{k=1}^{n}c_{k}\left(\ln\frac{t}{a}\right)^{\mu-k}-\frac{1}{\Gamma(\mu)}\int_{a}^{t}\left(\ln\frac{t}{\tau}\right)^{\mu-1}y(\tau)\frac{d\tau}{\tau},\hspace{0.4cm}t\geq a,
\end{align*}
where $c_{k}\in\mathbb{R}$, $k=1,2,\cdots,n$.
\end{lem}

In the following lemma we formulate the Green's function.

\begin{lem}\label{linearbvp}
For $\mu\in(2,\,3)$, $y\in L[a,\,b]$, the Hadamard fractional BVP
\begin{equation}\label{hlbvp}\begin{split}
&{}^{H}D_{a^{+}}^{\,\mu}x(t)+y(t)=0,\hspace{0.4cm}t\in(a,\,b),\\
&x(a)=a\,x'(a)=x(b)=0
\end{split}\end{equation}
has a solution
\begin{equation}\label{hlbvpsol}
x(t)=\int_{a}^{b}G(t,\tau)y(\tau)\frac{d\tau}{\tau},\hspace{0.4cm}t\in[a,\,b],
\end{equation}
where the Green's function $G$ is defined as
\begin{equation}\label{greens}
G(t,\tau)=\frac{1}{\Gamma(\mu)\left(\ln\frac{b}{a}\right)^{\mu-1}}\begin{cases}\left(\ln\frac{t}{a}\right)^{\mu-1}\left(\ln\frac{b}{\tau}\right)^{\mu-1}-\left(\ln\frac{t}{\tau}\right)^{\mu-1}\left(\ln\frac{b}{a}\right)^{\mu-1},\hspace{0.4cm}&a\leq\tau\leq t\leq b,\\
\left(\ln\frac{t}{a}\right)^{\mu-1}\left(\ln\frac{b}{\tau}\right)^{\mu-1},\hspace{0.4cm}&a\leq t\leq\tau\leq b.
\end{cases}
\end{equation}
\end{lem}

\begin{proof}
The Hadamard fractional differential equation \eqref{hlbvp} has a general solution
\begin{align*}
x(t)=\sum_{k=1}^{3}c_{k}\left(\ln\frac{t}{a}\right)^{\mu-k}-\frac{1}{\Gamma(\mu)}\int_{a}^{t}\left(\ln\frac{t}{\tau}\right)^{\mu-1}y(\tau)\frac{d\tau}{\tau},\hspace{0.4cm}t\in[a,\,b],
\end{align*}
where $c_{i}\in\mathbb{R}$, $i=1,2,3$. Now, using $x(a)=a\,x'(a)=x(b)=0$, we have
\begin{align*}
c_{3}=0,\,c_{2}=0,\,c_{1}=\frac{1}{\Gamma(\mu)\left(\ln\frac{b}{a}\right)^{\mu-1}}\int_{a}^{b}\left(\ln\frac{b}{\tau}\right)^{\mu-1}y(\tau)\frac{d\tau}{\tau}.
\end{align*}
So,
\begin{align*}
x(t)=\frac{1}{\Gamma(\mu)\left(\ln\frac{b}{a}\right)^{\mu-1}}\int_{a}^{b}\left(\ln\frac{t}{a}\right)^{\mu-1}\left(\ln\frac{b}{\tau}\right)^{\mu-1}y(\tau)\frac{d\tau}{\tau}-\frac{1}{\Gamma(\mu)}\int_{a}^{t}\left(\ln\frac{t}{\tau}\right)^{\mu-1}y(\tau)\frac{d\tau}{\tau},\hspace{0.4cm}t\in[a,\,b],
\end{align*}
which is equivalent to \eqref{hlbvpsol}.
\end{proof}

In the following lemma we provide some important properties of the Green's function \eqref{greens}.

\begin{lem}\label{properties}
For $t,\tau,s\in[a,\,b]$, we have
\begin{itemize}
\item[$(\mathbf{i})$] $G(t,\tau)\leq\frac{u(t)}{\Gamma(\mu)\left(\ln\frac{b}{a}\right)}$,
\item[$(\mathbf{ii})$] $G(t,\tau)\leq\frac{v(\tau)}{\Gamma(\mu)\left(\ln\frac{b}{a}\right)}$,
\item[$(\mathbf{iii})$] $G(t,\tau)\geq\frac{u(t)\,v(\tau)}{\Gamma(\mu)\left(\ln\frac{b}{a}\right)^{\mu+1}}$,
\item[$(\mathbf{iv})$] $G(t,\tau)\geq w(t)\,G(s,\tau)$,
\end{itemize}
where
\begin{align*}
u(t)=\left(\ln\frac{t}{a}\right)^{\mu-1}\left(\ln\frac{b}{t}\right),\hspace{0.4cm}v(t)=\left(\ln\frac{t}{a}\right)\left(\ln\frac{b}{t}\right)^{\mu-1},\hspace{0.4cm}w(t)=\frac{u(t)}{\left(\ln\frac{b}{a}\right)^{\mu}}.
\end{align*}
\end{lem}

\begin{proof}
$(\mathbf{i})$ For $a\leq\tau\leq t\leq b$, we have
\begin{align*}
G(t,\tau)=&\frac{1}{\Gamma(\mu)\left(\ln\frac{b}{a}\right)^{\mu-1}}\left(\left(\ln\frac{t}{a}\right)^{\mu-1}\left(\ln\frac{b}{\tau}\right)^{\mu-1}-\left(\ln\frac{t}{\tau}\right)^{\mu-1}\left(\ln\frac{b}{a}\right)^{\mu-1}\right)\\
=&\frac{1}{\Gamma(\mu)\left(\ln\frac{b}{a}\right)^{\mu-1}}\left(\left(\ln\frac{t}{a}\right)^{\mu-1}\left(\ln\frac{b}{\tau}\right)^{\mu-1}-\left(\ln\frac{t}{\tau}\right)^{\mu-2}\left(\ln\frac{t}{\tau}\right)\left(\ln\frac{b}{a}\right)^{\mu-1}\right).
\end{align*}
However, $G(t,\tau)$ is maximum along $\tau\in[a,\,b]$ satisfying $\frac{\partial G}{\partial\tau}(t,\tau)=0$, which implies that
\begin{align*}
\left(\ln\frac{t}{\tau}\right)^{\mu-2}\left(\ln\frac{b}{a}\right)^{\mu-1}=\left(\ln\frac{t}{a}\right)^{\mu-1}\left(\ln\frac{b}{\tau}\right)^{\mu-2}.
\end{align*}
Therefore,
\begin{align*}
G(t,\tau)\leq&\frac{1}{\Gamma(\mu)\left(\ln\frac{b}{a}\right)^{\mu-1}}\left(\left(\ln\frac{t}{a}\right)^{\mu-1}\left(\ln\frac{b}{\tau}\right)^{\mu-1}-\left(\ln\frac{t}{a}\right)^{\mu-1}\left(\ln\frac{b}{\tau}\right)^{\mu-2}\left(\ln\frac{t}{\tau}\right)\right)\\
=&\frac{1}{\Gamma(\mu)\left(\ln\frac{b}{a}\right)^{\mu-1}}\left(\left(\ln\frac{b}{\tau}\right)^{\mu-2}\left(\ln\frac{t}{a}\right)^{\mu-1}\left(\left(\ln\frac{b}{\tau}\right)-\left(\ln\frac{t}{\tau}\right)\right)\right)\\
\leq&\frac{1}{\Gamma(\mu)\left(\ln\frac{b}{a}\right)^{\mu-1}}\left(\left(\ln\frac{b}{a}\right)^{\mu-2}\left(\ln\frac{t}{a}\right)^{\mu-1}\left(\left(\ln\frac{b}{\tau}\right)-\left(\ln\frac{t}{\tau}\right)\right)\right)\\
\leq&\frac{1}{\Gamma(\mu)\left(\ln\frac{b}{a}\right)^{\mu-1}}\left(\ln\frac{b}{a}\right)^{\mu-2}\left(\ln\frac{t}{a}\right)^{\mu-1}\left(\ln\frac{b}{t}\right)\\
=&\frac{u(t)}{\Gamma(\mu)\left(\ln\frac{b}{a}\right)}.
\end{align*}
Also, for $a\leq t\leq\tau\leq b$, we have
\begin{align*}
G(t,\tau)=&\frac{1}{\Gamma(\mu)\left(\ln\frac{b}{a}\right)^{\mu-1}}\left(\ln\frac{t}{a}\right)^{\mu-1}\left(\ln\frac{b}{\tau}\right)^{\mu-1}\\
\leq&\frac{1}{\Gamma(\mu)\left(\ln\frac{b}{a}\right)^{\mu-1}}\left(\ln\frac{t}{a}\right)^{\mu-1}\left(\ln\frac{b}{\tau}\right)^{\mu-2}\left(\ln\frac{b}{\tau}\right)\\
\leq&\frac{1}{\Gamma(\mu)\left(\ln\frac{b}{a}\right)^{\mu-1}}\left(\ln\frac{b}{a}\right)^{\mu-2}\left(\ln\frac{t}{a}\right)^{\mu-1}\left(\ln\frac{b}{t}\right)\\
=&\frac{u(t)}{\Gamma(\mu)\left(\ln\frac{b}{a}\right)}.
\end{align*}
$(\mathbf{ii})$ Again, for $a\leq\tau\leq t\leq b$, we have
\begin{align*}
G(t,\tau)=&\frac{1}{\Gamma(\mu)\left(\ln\frac{b}{a}\right)^{\mu-1}}\left(\left(\ln\frac{t}{a}\right)^{\mu-1}\left(\ln\frac{b}{\tau}\right)^{\mu-1}-\left(\ln\frac{t}{\tau}\right)^{\mu-1}\left(\ln\frac{b}{a}\right)^{\mu-1}\right)\\
=&\frac{1}{\Gamma(\mu)\left(\ln\frac{b}{a}\right)^{\mu-1}}\left(\left(\ln\frac{t}{a}\right)^{\mu-1}\left(\ln\frac{b}{\tau}\right)^{\mu-1}-\left(\ln\frac{t}{\tau}\right)^{\mu-2}\left(\ln\frac{t}{\tau}\right)\left(\ln\frac{b}{a}\right)^{\mu-1}\right).
\end{align*}
However, $G(t,\tau)$ is maximum along $t\in[a,\,b]$ satisfying $\frac{\partial G}{\partial t}(t,\tau)=0$, which implies that
\begin{align*}
\left(\ln\frac{t}{\tau}\right)^{\mu-2}\left(\ln\frac{b}{a}\right)^{\mu-1}=\left(\ln\frac{t}{a}\right)^{\mu-2}\left(\ln\frac{b}{\tau}\right)^{\mu-1}.
\end{align*}
Therefore,
\begin{align*}
G(t,\tau)\leq&\frac{1}{\Gamma(\mu)\left(\ln\frac{b}{a}\right)^{\mu-1}}\left(\left(\ln\frac{t}{a}\right)^{\mu-1}\left(\ln\frac{b}{\tau}\right)^{\mu-1}-\left(\ln\frac{t}{a}\right)^{\mu-2}\left(\ln\frac{b}{\tau}\right)^{\mu-1}\left(\ln\frac{t}{\tau}\right)\right)\\
=&\frac{1}{\Gamma(\mu)\left(\ln\frac{b}{a}\right)^{\mu-1}}\left(\left(\ln\frac{t}{a}\right)^{\mu-2}\left(\ln\frac{b}{\tau}\right)^{\mu-1}\left(\left(\ln\frac{t}{a}\right)-\left(\ln\frac{t}{\tau}\right)\right)\right)\\
\leq&\frac{1}{\Gamma(\mu)\left(\ln\frac{b}{a}\right)^{\mu-1}}\left(\left(\ln\frac{b}{a}\right)^{\mu-2}\left(\ln\frac{b}{\tau}\right)^{\mu-1}\left(\ln\frac{\tau}{a}\right)\right)\\
\leq&\frac{1}{\Gamma(\mu)\left(\ln\frac{b}{a}\right)^{\mu-1}}\left(\ln\frac{b}{a}\right)^{\mu-2}\left(\ln\frac{\tau}{a}\right)\left(\ln\frac{b}{\tau}\right)^{\mu-1}\\
=&\frac{v(\tau)}{\Gamma(\mu)\left(\ln\frac{b}{a}\right)}.
\end{align*}
Also, for $a\leq t\leq\tau\leq b$, we have
\begin{align*}
G(t,\tau)=&\frac{1}{\Gamma(\mu)\left(\ln\frac{b}{a}\right)^{\mu-1}}\left(\ln\frac{t}{a}\right)^{\mu-1}\left(\ln\frac{b}{\tau}\right)^{\mu-1}\\
\leq&\frac{1}{\Gamma(\mu)\left(\ln\frac{b}{a}\right)^{\mu-1}}\left(\ln\frac{t}{a}\right)^{\mu-2}\left(\ln\frac{t}{a}\right)\left(\ln\frac{b}{\tau}\right)^{\mu-1}\\
\leq&\frac{1}{\Gamma(\mu)\left(\ln\frac{b}{a}\right)^{\mu-1}}\left(\ln\frac{b}{a}\right)^{\mu-2}\left(\ln\frac{\tau}{a}\right)\left(\ln\frac{b}{\tau}\right)^{\mu-1}\\
=&\frac{v(\tau)}{\Gamma(\mu)\left(\ln\frac{b}{a}\right)}.
\end{align*}
$(\mathbf{iii})$ For $a\leq\tau\leq t\leq b$, we have
\begin{align*}
G(t,\tau)=&\frac{1}{\Gamma(\mu)\left(\ln\frac{b}{a}\right)^{\mu-1}}\left(\left(\ln\frac{t}{a}\right)^{\mu-1}\left(\ln\frac{b}{\tau}\right)^{\mu-1}-\left(\ln\frac{t}{\tau}\right)^{\mu-1}\left(\ln\frac{b}{a}\right)^{\mu-1}\right)\\
=&\frac{1}{\Gamma(\mu)\left(\ln\frac{b}{a}\right)^{\mu+1}}\left(\left(\ln\frac{b}{a}\right)^{2}\left(\ln\frac{t}{a}\right)^{\mu-1}\left(\ln\frac{b}{\tau}\right)^{\mu-1}-\left(\ln\frac{b}{a}\right)^{2}\left(\ln\frac{t}{\tau}\right)^{\mu-1}\left(\ln\frac{b}{a}\right)^{\mu-1}\right)\\
\geq&\frac{1}{\Gamma(\mu)\left(\ln\frac{b}{a}\right)^{\mu+1}}\left(\left(\ln\frac{t}{a}\right)^{\mu}\left(\ln\frac{b}{\tau}\right)^{\mu}-\left(\ln\frac{b}{a}\right)^{\mu+1}\left(\ln\frac{t}{\tau}\right)^{\mu-1}\right)\\
=&\frac{1}{\Gamma(\mu)\left(\ln\frac{b}{a}\right)^{\mu+1}}\left(\ln\frac{t}{a}\right)^{\mu-1}\left(\ln\frac{b}{\tau}\right)^{\mu-1}\left(\ln\frac{b}{\tau}\right)\left(\left(\ln\frac{t}{a}\right)-\frac{\left(\ln\frac{b}{a}\right)^{\mu+1}\left(\ln\frac{t}{\tau}\right)^{\mu-1}}{\left(\ln\frac{b}{\tau}\right)\left(\ln\frac{t}{a}\right)^{\mu-1}\left(\ln\frac{b}{\tau}\right)^{\mu}}\right)\\
\geq&\frac{1}{\Gamma(\mu)\left(\ln\frac{b}{a}\right)^{\mu+1}}\left(\ln\frac{t}{a}\right)^{\mu-1}\left(\ln\frac{b}{\tau}\right)^{\mu-1}\left(\ln\frac{b}{t}\right)\left(\left(\ln\frac{t}{a}\right)-\frac{\left(\ln\frac{b}{a}\right)^{\mu+1}\left(\ln\frac{t}{\tau}\right)^{\mu-1}}{\left(\ln\frac{b}{\tau}\right)\left(\ln\frac{t}{a}\right)^{\mu-1}\left(\ln\frac{b}{\tau}\right)^{\mu}}\right).
\end{align*}
Now using $\theta(t):=\left(\ln\frac{t}{a}\right)-\frac{\left(\ln\frac{b}{a}\right)^{\mu+1}\left(\ln\frac{t}{\tau}\right)^{\mu-1}}{\left(\ln\frac{b}{\tau}\right)\left(\ln\frac{t}{a}\right)^{\mu-1}\left(\ln\frac{b}{\tau}\right)^{\mu}}\geq\theta(\tau)$, we have
\begin{align*}
G(t,\tau)\geq&\frac{1}{\Gamma(\mu)\left(\ln\frac{b}{a}\right)^{\mu+1}}\left(\ln\frac{t}{a}\right)^{\mu-1}\left(\ln\frac{b}{t}\right)\left(\ln\frac{\tau}{a}\right)\left(\ln\frac{b}{\tau}\right)^{\mu-1}\\
=&\frac{u(t)\,v(\tau)}{\Gamma(\mu)\left(\ln\frac{b}{a}\right)^{\mu+1}}.
\end{align*}
Also for $a\leq t\leq\tau\leq b$, we have
\begin{align*}
G(t,\tau)=&\frac{1}{\Gamma(\mu)\left(\ln\frac{b}{a}\right)^{\mu-1}}\left(\ln\frac{t}{a}\right)^{\mu-1}\left(\ln\frac{b}{\tau}\right)^{\mu-1}\\
=&\frac{1}{\Gamma(\mu)\left(\ln\frac{b}{a}\right)^{\mu-1}\left(\ln\frac{b}{t}\right)\left(\ln\frac{\tau}{a}\right)}\left(\ln\frac{t}{a}\right)^{\mu-1}\left(\ln\frac{b}{t}\right)\left(\ln\frac{\tau}{a}\right)\left(\ln\frac{b}{\tau}\right)^{\mu-1}\\
\geq&\frac{1}{\Gamma(\mu)\left(\ln\frac{b}{a}\right)^{\mu-1}\left(\ln\frac{b}{a}\right)^{2}}\left(\ln\frac{t}{a}\right)^{\mu-1}\left(\ln\frac{b}{t}\right)\left(\ln\frac{\tau}{a}\right)\left(\ln\frac{b}{\tau}\right)^{\mu-1}\\
=&\frac{u(t)\,v(\tau)}{\Gamma(\mu)\left(\ln\frac{b}{a}\right)^{\mu+1}}.
\end{align*}
$(\mathbf{iv})$ For $t,\tau,s\in[a,\,b]$, using $G(s,\tau)\leq\frac{v(\tau)}{\Gamma(\mu)\left(\ln\frac{b}{a}\right)}$, we have
\begin{align*}
G(t,\tau)\geq\frac{u(t)\,v(\tau)}{\Gamma(\mu)\left(\ln\frac{b}{a}\right)^{\mu+1}}\geq\frac{u(t)}{\left(\ln\frac{b}{a}\right)^{\mu}}G(s,\tau)=w(t)\,G(s,\tau).
\end{align*}
\end{proof}

For $x\in C[a,\,b]$, we write $\|x\|=\max_{t\in[a,\,b]}|x(t)|$. Clearly, $(C[a,\,b],\|\cdot\|)$ is a Banach space. For $r>0$, $\Omega_{r}:=\{x\in C[a,\,b]:\|x\|<r\}$ is a bounded and open subset of $C[a,\,b]$. Moreover, $P:=\left\{x\in C[a,\,b]:x(t)\geq w(t)\,\|x\|\text{ for all }t\in[a,\,b]\right\}$ is a cone of $C[a,\,b]$.

Throughout this article, assume that the following holds:

\begin{itemize}
\item[$(\mathbf{A}_{1})$] $\int_{a}^{b}f\left(t,c\,w(t)\right)dt<\infty$ for $c>0$.
\item[$(\mathbf{A}_{2})$] $f\in C((a,\,b)\times(0,\infty),(0,\infty))$ and there exist $(\vartheta,\rho)\in(a,\,b)\times(0,\infty)$ such that
    \begin{align*}
    f(\vartheta,\rho)\leq f(t,\rho)\leq f(t,x),\hspace{0.4cm}(t,x)\in(a,\,b)\times(0,\infty).
    \end{align*}
    Moreover, for each $t\in(a,\,b)$, $f(t,\cdot)$ is decreasing on $(0,\,\rho)$.
\item[$(\mathbf{A}_{3})$]
\begin{align*}
\sup_{\kappa\in(0,\,\rho)}\,\frac{\kappa}{\int_{a}^{b}v(t)\,f\left(t,\kappa\,w(t)\right)\frac{dt}{t}}>\frac{1}{\Gamma(\mu)\left(\ln\frac{b}{a}\right)}.
\end{align*}
\end{itemize}
In view of $(\mathbf{A}_{3})$, there exist $R>0$ such that
\begin{align*}
\frac{R}{\int_{a}^{b}v(t)\,f\left(t,R\,w(t)\right)\frac{dt}{t}}>\frac{1}{\Gamma(\mu)\left(\ln\frac{b}{a}\right)}.
\end{align*}
So, we can choose $\varepsilon>0$ such that
\begin{equation}\label{eps}
\frac{R-\varepsilon}{\int_{a}^{b}v(t)\,f\left(t,(R-\varepsilon)\,w(t)\right)\frac{dt}{t}}\geq\frac{1}{\Gamma(\mu)\left(\ln\frac{b}{a}\right)}.
\end{equation}
Choose $n_{0}\in\mathbb{N}$ such that $\frac{1}{n_{0}}<\varepsilon$. For $n\in\{n_{0},n_{0}+1,n_{0}+2,\cdots\}$, define a map $T_{n}:P\rightarrow P$ as
\begin{equation}\label{maptn}
(T_{n}x)(t)=\int_{a}^{b}G(t,\tau)f\left(\tau,|x(\tau)|+\frac{1}{n}\right)\frac{d\tau}{\tau},\hspace{0.4cm}t\in[a,\,b].
\end{equation}

\begin{lem}
The map $T_{n}:\overline{\Omega}_{R-\varepsilon}\cap P\rightarrow P$ is completely continuous.
\end{lem}

\begin{proof}
For $x\in P$, $t\in[a,\,b]$, from \eqref{maptn}, using Lemma \ref{properties} $(\mathbf{iv})$, we have
\begin{align*}
(T_{n}x)(t)=&\int_{a}^{b}G(t,\tau)f\left(\tau,|x(\tau)|+\frac{1}{n}\right)\frac{d\tau}{\tau}\\
\geq&w(t)\int_{a}^{b}G(s,\tau)f\left(\tau,|x(\tau)|+\frac{1}{n}\right)\frac{d\tau}{\tau},\hspace{0.4cm}s\in[a,\,b]\\
=&w(t)\,(T_{n}x)(s),\hspace{4.55cm}s\in[a,\,b]
\end{align*}
which implies that
\begin{align*}
(T_{n}x)(t)\geq w(t)\,\|T_{n}x\|,\hspace{0.4cm}t\in[a,\,b],
\end{align*}
which implies that $T_{n}P\subset P$. Moreover, $T_{n}:\overline{\Omega}_{R-\varepsilon}\cap P\rightarrow P$ is continuous and compact.
\end{proof}
Also, we need the following fixed point index result \cite{guo} for our main theorem.
\begin{lem}\label{lemindexone}\cite{guo}
Assume that $T:\overline{\Omega}_{r}\cap P\rightarrow P$ is a completely continuous map such that $\|Tx\|\leq\|x\|$ for $x\in\partial\Omega_{r}\cap P$. Then, the fixed point index $\ind(T,\overline{\Omega}_{r}\cap P,P)=1$.
\end{lem}

\section{Main Result}\label{main}

\begin{thm}\label{mainth}
Assume that $(\mathbf{A}_{1})-(\mathbf{A}_{3})$ hold. Then the Hadamard fractional SBVP \eqref{hbvp} has a positive solution.
\end{thm}

\begin{proof}
For $x\in\partial\Omega_{R-\varepsilon}\cap P$, we have $(R-\varepsilon)\,w(t)\leq x(t)\leq R-\varepsilon$ for $t\in[a,\,b]$. Therefore, in view of \eqref{maptn}, using Lemma \ref{properties} $(\mathbf{ii})$ and \eqref{eps}, for $t\in[a,\,b]$, we have
\begin{align*}
(T_{n}x)(t)=&\int_{a}^{b}G(t,\tau)f\left(\tau,|x(\tau)|+\frac{1}{n}\right)\frac{d\tau}{\tau}\\
\leq&\frac{1}{\Gamma(\mu)\left(\ln\frac{b}{a}\right)}\int_{a}^{b}v(\tau)f\left(\tau,|x(\tau)|+\frac{1}{n}\right)\frac{d\tau}{\tau}\\
\leq&\frac{1}{\Gamma(\mu)\left(\ln\frac{b}{a}\right)}\int_{a}^{b}v(\tau)f\left(\tau,(R-\varepsilon)w(\tau)\right)\frac{d\tau}{\tau}\\
\leq&R-\varepsilon,
\end{align*}
which implies that
\begin{align*}
\|T_{n}x\|\leq\|x\|,\hspace{0.4cm}x\in\partial\Omega_{R-\varepsilon}\cap P,
\end{align*}
which in view of Lemma \ref{lemindexone}, leads to
\begin{align*}
\ind(T_{n},\overline{\Omega}_{R-\varepsilon}\cap P,P)=1.
\end{align*}
So, there exist $x_{n}\in\Omega_{R-\varepsilon}\cap P$ such that $T_{n}x_{n}=x_{n}$. Moreover, using \eqref{maptn}, Lemma \ref{properties} $(\mathbf{iii})$, we have
\begin{align*}
x_{n}(t)=&\int_{a}^{b}G(t,\tau)f\left(\tau,|x_{n}(\tau)|+\frac{1}{n}\right)\frac{d\tau}{\tau}\\
\geq&\frac{u(t)}{\Gamma(\mu)\left(\ln\frac{b}{a}\right)^{\mu+1}}\int_{a}^{b}v(\tau)f\left(\vartheta,\rho\right)\frac{d\tau}{\tau}\\
=&\frac{w(t)f(\vartheta,\rho)}{\Gamma(\mu+2)}\left(\ln\frac{b}{a}\right)^{\mu}\\
=&r\,w(t),\text{ where }r=\frac{f(\vartheta,\rho)}{\Gamma(\mu+2)}\left(\ln\frac{b}{a}\right)^{\mu}.
\end{align*}
Consequently, $x_{n}\in P$ satisfy
\begin{equation}\label{xnt}x_{n}(t)=\int_{a}^{b}G(t,\tau)f\left(\tau,x_{n}(\tau)+\frac{1}{n}\right)\frac{d\tau}{\tau},\hspace{0.4cm}t\in[a,\,b],\end{equation}
and
\begin{align*}
r\,w(t)\leq x_{n}(t)\leq R-\varepsilon,\hspace{0.4cm}t\in[a,\,b],
\end{align*}
which shows that the sequence $\{x_{n}\}_{n=n_{0}}^{\infty}$ is uniformly bounded on $[a,\,b]$. Moreover, since the Green's function \eqref{greens} is uniformly continuous on $[a,\,b]\times[a,\,b]$, the sequence $\{x_{n}\}_{n=n_{0}}^{\infty}$ is equicontinuous on $[a,\,b]$. Thus by Arzel\`{a}-Ascoli theorem the sequence $\{x_{n}\}_{n=n_{0}}^{\infty}$ is relatively compact and consequently there exist a subsequence $\{x_{n_{k}}\}_{k=1}^{\infty}$ converging uniformly to $x\in C[a,\,b]$. Moreover, in view of \eqref{xnt}, we have
\begin{align*}
x_{n_{k}}(t)=\int_{a}^{b}G(t,\tau)f\left(\tau,x_{n_{k}}(\tau)+\frac{1}{n_{k}}\right)\frac{d\tau}{\tau},
\end{align*}
as $k\rightarrow\infty$, in view of the Lebesgue dominated convergence theorem, we obtain
\begin{equation}\label{intsol}
x(t)=\int_{a}^{b}G(t,\tau)f(\tau,x(\tau))\frac{d\tau}{\tau},
\end{equation}
which in view of Lemma \ref{linearbvp}, leads to
\begin{align*}
{}^{H}D_{a^{+}}^{\mu}x(t)+f(t,x(t))&=0,\hspace{0.4cm}t\in(a,\,b)\\
x(a)=a\,x'(a)=x(b)&=0.
\end{align*}
Also, ${}^{H}D_{a^{+}}^{\mu}x\in C(a,\,b)$. Further, from \eqref{intsol} in view of $(\mathbf{A}_{2})$ and using Lemma \ref{properties} $(\mathbf{iii})$, we have
\begin{align*}
x(t)=\int_{a}^{b}G(t,\tau)f(\tau,x(\tau))\frac{d\tau}{\tau}\geq r\,w(t),\hspace{0.4cm}t\in[a,\,b],
\end{align*}
which shows that $x(t)>0$ for $t\in(a,\,b)$. Hence $x\in C[a,\,b]$ with ${}^{H}D_{a^{+}}^{\mu}x\in C(a,\,b)$ is a positive solution of the Hadamard fractional SBVP \eqref{hbvp}.
\end{proof}

\begin{ex}
\begin{equation}\label{ebvp}\begin{split}
{}^{H}D_{1^{+}}^{\,^{2.9}}\,x(t)+\frac{\lambda}{\sqrt[4]{\left(\ln t\right)^{1.9}\left(\ln\frac{e}{t}\right)}}\left(\sqrt{x(t)}+\frac{2}{\sqrt[4]{x(t)}}\right)&=0,\hspace{0.4cm}t\in(1,\,e),\\
x(1)=x'(1)=x(e)&=0,
\end{split}\end{equation}
where
\begin{align*}
0<\lambda<\sup_{\kappa\in(0,\,1)}\,\frac{\kappa\,\Gamma(2.9)}{\int_{1}^{e}\sqrt[4]{\left(\ln t\right)^{2.1}\left(\ln\frac{e}{t}\right)^{6.6}}\left(\sqrt{\kappa\,\left(\ln t\right)^{1.9}\left(\ln\frac{e}{t}\right)}+\frac{2}{\sqrt[4]{\kappa\,\left(\ln t\right)^{1.9}\left(\ln\frac{e}{t}\right)}}\right)\frac{dt}{t}}.
\end{align*}
Here
\begin{align*}
f(t,x)=\frac{\lambda}{\sqrt[4]{\left(\ln t\right)^{1.9}\left(\ln\frac{e}{t}\right)}}\left(\sqrt{x(t)}+\frac{2}{\sqrt[4]{x(t)}}\right)
\end{align*}
Clearly, $f:(1,\,e)\times(0,\infty)\rightarrow(0,\infty)$ is continuous and singular at $t=1$, $t=e$ and $x=0$. Further, for each $t\in(1,\,e)$, $f(t,\cdot)$ is decreasing on $(0,\,1)$, and
\begin{align*}
f(e^{\frac{19}{29}},1)\leq f(t,1)\leq f(t,x),\hspace{0.4cm}(t,x)\in(1,\,e)\times(0,\infty).
\end{align*}
Moreover,
\begin{align*}
\int_{1}^{e}f(t,c\,w(t))dt=\int_{1}^{e}f\left(t,c\,\left(\ln t\right)^{1.9}\left(\ln\frac{e}{t}\right)\right)dt<\infty\text{ for }c>0,
\end{align*}
\begin{align*}
\sup_{\kappa\in(0,\,1)}&\,\frac{\kappa}{\int_{1}^{e}v(t)\,f\left(t,\kappa\,w(t)\right)\frac{dt}{t}}\\
=\sup_{\kappa\in(0,\,1)}&\,\frac{\kappa}{\lambda\int_{1}^{e}\sqrt[4]{\left(\ln t\right)^{2.1}\left(\ln\frac{e}{t}\right)^{6.6}}\left(\sqrt{\kappa\,\left(\ln t\right)^{1.9}\left(\ln\frac{e}{t}\right)}+\frac{2}{\sqrt[4]{\kappa\,\left(\ln t\right)^{1.9}\left(\ln\frac{e}{t}\right)}}\right)\frac{dt}{t}}>\frac{1}{\Gamma(2.9)}.
\end{align*}
Hence, the assumptions $(\mathbf{A}_{1})-(\mathbf{A}_{3})$ are satisfied. Therefore, by Theorem \ref{mainth}, the Hadamard fractional SBVP \eqref{ebvp} has a positive solution.
\end{ex}

\end{document}